\documentclass[11pt,a4paper]{amsart}
\usepackage{color}
\usepackage[T1]{fontenc}
\usepackage[utf8]{inputenc}
\usepackage{hyperref,amssymb,url,upref,verbatim,xspace,mathrsfs}
\RequirePackage[all,ps,cmtip]{xy}
\newdir^{ (}{{}*!/-5pt/\dir^{(}}
\newdir_{ (}{{}*!/-5pt/\dir_{(}}
\RequirePackage[mathscr]{eucal}
\usepackage{mathrsfs}\let\mathcal\mathscr
\definecolor{darkgreen}{rgb}{.0,.41961,.18431}

\theoremstyle{plain}

\newtheorem*{theorem*}{Main Theorem}
\newtheorem{theorem}{Theorem}
\newtheorem{proposition}[theorem]{Proposition}
\newtheorem{lemma}[theorem]{Lemma}
\newtheorem{corollary}[theorem]{Corollary}

\theoremstyle{definition}
\newtheorem{definition}[theorem]{Definition}
\newtheorem{remark}[theorem]{Remark}

\newtheorem{notation}[theorem]{Notation}
\newtheorem{example}[theorem]{Example}
\newtheorem{properties}[theorem]{Properties}

\renewcommand{\theenumi}{\alph{enumi}}

\AtBeginDocument{%
}

\newcounter{toto}
\def\thetoto{\arabic{toto}}
\let\oldmarginpar\marginpar
\def\marginpar#1{\refstepcounter{toto}\textsuperscript{\textup{[\thetoto]}}\oldmarginpar{\footnotesize\textsuperscript{[\thetoto]}\,#1}}

\def\mainmatter{\renewcommand{\baselinestretch}{1.1}\normalfont}
\def\backmatter{\renewcommand{\baselinestretch}{1}\normalfont}

\makeatletter
\@namedef{subjclassname@2010}{\textup{2010} Mathematics Subject Classification}
\def\l@section{\@tocline{1}{0pt}{0pc}{}{}}
\def\l@subsection{\@tocline{2}{0pt}{1.5pc}{}{}}
\def\l@subsubsection{\@tocline{3}{0pt}{2pc}{}{}}
\makeatother

\newcommand{\C}{\mathbb{C}}\let\CC\C
\let\NN\N
\newcommand{\R}{\mathbb{R}}\let\RR\R

\newcommand{\Z}{\mathbb{Z}}

\newcommand{\bD}{\boldsymbol{D}}
\newcommand{\shhom}{\mathcal{H}\!\mathit{om}}\let\ho\shhom
\DeclareMathOperator{\conj}{c}
\DeclareMathOperator{\Conj}{C}
\let\Rhom\rh

\DeclareMathOperator{\RH}{RH}

\newcommand{\rb}{\mathrm{b}}
\newcommand{\rc}{\mathrm{c}}
\newcommand{\rd}{\mathrm{d}}

\newcommand{\coh}{\mathrm{coh}}
\newcommand{\hol}{\mathrm{hol}}
\newcommand{\rhol}{\mathrm{rhol}}

\newcommand{\Mod}{\mathrm{Mod}}

\newcommand{\op}{\mathrm{op}}
\newcommand{\ord}{\mathrm{ord}}

\newcommand{\cc}{{\C\textup{-c}}}

\newcommand{\XC}{X\times\CC}
\newcommand{\XS}{X\times S}
\newcommand{\XbS}{\ov X\times S}

\newcommand{\XpS}{X'\times S}

\newcommand{\YS}{Y\times S}

\newcommand{\DX}{\shd_X}
\newcommand{\DXb}{\shd_{\ov X}}
\newcommand{\DXC}{\shd_{\XC/\CC}}
\newcommand{\DXS}{\shd_{\XS/S}}
\newcommand{\DXbS}{\shd_{\XbS/S}}

\newcommand{\DXSp}{\shd_{\XS'/S'}}

\newcommand{\DYS}{\shd_{Y\times S/S}}

\DeclareMathOperator{\Char}{Char}

\DeclareMathOperator{\rD}{\mathsf{D}}

\DeclareMathOperator{\DR}{DR}
\DeclareMathOperator{\Cb}{\mathfrak{C}}
\DeclareMathOperator{\Db}{\mathfrak{Db}}
\DeclareMathOperator{\pDR}{{}^\mathrm{p}DR}

\DeclareMathOperator{\id}{Id}

\DeclareMathOperator{\reel}{Re}

\DeclareMathOperator{\pSol}{{}^\mathrm{p}Sol}

\let\bar\overline

\let\ov\overline
\let\epsilon\varepsilon

\let\setminus\smallsetminus
\let\leq\leqslant
\let\geq\geqslant

\def\loccit{loc.\kern3pt cit.{}\xspace}
\def\cf{cf.\kern.3em}
\def\Cf{Cf.\kern.3em}
\def\eg{e.g.\kern.3em}
\def\ie{i.e.,\ }
\def\resp{\text{resp.}\kern.3em}
\let\moins\smallsetminus

\newcommand{\Df}{{}_{\scriptscriptstyle\mathrm{D}}f}

\newcommand{\Di}{{}_{\scriptscriptstyle\mathrm{D}}i}
\newcommand{\Drho}{{}_{\scriptscriptstyle\mathrm{D}}\rho}

\newcommand{\Dpi}{{}_{\scriptscriptstyle\mathrm{D}}\pi}
\newcommand{\cbbullet}{{\raisebox{1pt}{$\sbullet$}}}
\newcommand{\sbullet}{{\scriptscriptstyle\bullet}}
\newcommand{\pOS}{p^{-1}\sho_S}

\def\shc{\mathcal{C}}
\def\shd{\mathcal{D}}

\let\cF F
\let\cG G
\def\shh{\mathcal{H}}
\def\shi{\mathcal{I}}
\def\shj{\mathcal{J}}

\def\shm{\mathcal{M}}
\def\shn{\mathcal{N}}
\def\sho{\mathcal{O}}

\def\shx{\mathcal{X}}

\newcommand{\RedefinitSymbole}[1]{%
\expandafter\let\csname old\string#1\endcsname=#1
\let#1=\relax
\newcommand{#1}{\csname old\string#1\endcsname\,}%
}
\RedefinitSymbole{\forall} \RedefinitSymbole{\exists}

\def\to{\mathchoice{\longrightarrow}{\rightarrow}{\rightarrow}{\rightarrow}}
\def\mto{\mathchoice{\longmapsto}{\mapsto}{\mapsto}{\mapsto}}
\def\hto{\mathrel{\lhook\joinrel\to}}

\def\To#1{\mathchoice{\xrightarrow{\textstyle\kern4pt#1\kern3pt}}{\stackrel{#1}{\longrightarrow}}{}{}}

\def\isom{\stackrel{\sim}{\longrightarrow}}

\let\oldbigoplus\bigoplus
\renewcommand{\bigoplus}{\mathop{\textstyle\oldbigoplus}\displaylimits}
\let\oldbigwedge\bigwedge
\renewcommand{\bigwedge}{\mathop{\textstyle\oldbigwedge}\displaylimits}
\let\oldbigcap\bigcap
\renewcommand{\bigcap}{\mathop{\textstyle\oldbigcap}\displaylimits}
\let\oldbigcup\bigcup
\renewcommand{\bigcup}{\mathop{\textstyle\oldbigcup}\displaylimits}

\newcommand{\bmm}{\boldsymbol{m}}
\newcommand{\bmr}{\boldsymbol{r}}

\DeclareMathOperator{\Supp}{Supp}

\begin{document}

\author[T. Monteiro Fernandes]{Teresa Monteiro Fernandes}
\address[T. Monteiro Fernandes]{Centro de Matem\'atica, Aplica\c{c}\~{o}es Funda\-men\-tais e Investiga\c c\~ao Operacional and Departamento de Matem\' atica da Faculdade de Ci\^en\-cias da Universidade de Lisboa, Bloco C6, Piso 2, Campo Grande, 1749-016, Lisboa
Portugal}
\email{mtfernandes@fc.ul.pt}

\author[C.~Sabbah]{Claude Sabbah}
\address[C.~Sabbah]{CMLS, CNRS, École polytechnique, Institut Polytechnique de Paris, 91128 Palaiseau cedex, France}
\email{Claude.Sabbah@polytechnique.edu}
\urladdr{http://www.math.polytechnique.fr/perso/sabbah}

\title{The relative hermitian duality functor}

\date{\today}



\begin{abstract}
We extend to the category of relative regular holonomic modules on a manifold $X$, parametrized by a curve $S$, the Hermitian duality functor of Kashiwara. We prove that this functor is an equivalence with the similar category on the conjugate manifold $\overline X$, parametrized by the same curve.
\end{abstract}

\maketitle
\tableofcontents
\mainmatter

\vspace*{-2\baselineskip}\vskip0pt%
\section{Introduction and statement of the main result}

Let $X$ and $S$ be complex manifolds, $S$ having dimension at most one. We~denote by $p_X:\XS\to S$ (or simply by $p$) the projection. The ring of holomorphic differential operators on $\XS$ relative to the projection $p$ is denoted by $\DXS$. It is a subsheaf of the sheaf of holomorphic differential operators $\shd_{\XS}$.

We have introduced in \cite{MFCS2,FMFS19} the notion of \emph{regular holonomic $\DXS$\nobreakdash-mod\-ule} and the category $\rD^\rb_\rhol(\DXS)$ of bounded complexes of $\DXS$\nobreakdash-mod\-ules having regular holonomic cohomology: a coherent $\DXS$-mod\-ule~$\shm$ is said to be regular holonomic if it is holonomic, that is, its characteristic variety is contained in a product $\Lambda\times S$, $\Lambda$ being analytic lagrangian $\C^*$-homogeneous in $T^*X$, and the restriction (in the sense of $\sho$-modules) to each fiber of $p$, denoted by $Li^*_{s}\shm$ for each $s\in S$, belongs to $\rD^\rb_{\rhol}(\shd_X)$. We have shown that the category $\rD^\rb_{\rhol}(\shd_{\bullet\times S/S})$ is stable by the following functors.
\begin{enumerate}
\item
$\Df_*$ for a projective morphism $f:Y\to X$, or for a proper morphism and restricting to $f$-good objects (\cite[Cor.\,2.4]{MFCS2}, where the condition $f$-good or projective was mistakenly forgotten, \cf\loccit, Th.\,1.17).
\item
$\Df^*$ for any holomorphic map $f:Y\to X$ (\cf\cite[Th.\,2]{FMFS19}).
\item
Duality $\bD$ (\cite[Prop.\,2.3]{MFCS2}).
\item
$R\Gamma_{[Z\times S]}$ for any $Z$ closed analytic subset of $X$ (this follows from \cite[Prop.\,2.6(b)]{FMFS19} and the relative Riemann-Hilbert correspondence of \loccit).
\end{enumerate}

The pair of quasi-inverse contravariant functors $(\pSol,\RH^S)$, providing the relative Riemann-Hilbert correspondence of \loccit\ ($\pSol$ denotes the derived solution functor shifted by the dimension of $X$), makes the above functors compatible with the corresponding functors on the category $\rD^\rb_\cc(\pOS)$ of $S$-$\CC$-constructible complexes. This is the triangulated category of complexes $F$ of sheaves of $p^{-1}\sho_S$-modules such that, for a suitable $\C$-analytical stratification $(X_{\alpha})_{\alpha\in A}$ of $X$ and any $\alpha\in A$, each cohomology sheaf of $F|_{X_{\alpha}\times S}$~is, locally on $X_{\alpha}\times S$, isomorphic to $p^{-1} G$ for some coherent $\sho_S$-module $G$. It~was proved in \cite{MFCS1} and \cite{MFCS2} that for all $s\in S$ the restriction functor $Li^*_s$ is conservative in both categories $\rD^\rb_{\cc}(p^{-1}\sho_S)$ and $\rD^\rb_{\hol}(\DXS)$.\enlargethispage{-\baselineskip}%

In this article, extending some definitions and results of \cite{Kashiwara86} (\cf Notation~\ref{notation} below for the notion of conjugate manifold), we define the contravariant \emph{relative Hermitian duality functor} $\Conj_{X,\ov X}^S$ and its covariant companion, the \emph{relative conjugation functor} $\conj_{X,\ov X}^S:=\Conj_{X,\ov X}^S\circ \bD$, both from $\DXS$-modules to $\DXbS$-modules. For that purpose, we make use of the sheaf of relative distributions $\Db_{\XS/S}$, consisting of those distributions on $\XS$ which are holomorphic with respect to $S$ (\ie killed by the $\ov\partial_S$ operator). A reminder on this sheaf is given in the appendix, for the sake of being complete.

\begin{notation}\label{notation}
For a complex manifold $\shx$, endowed with its sheaf $\sho_\shx$ of holomorphic functions, we denote by $\ov\shx$ the same underlying $C^\infty$-manifold, endowed with the sheaf $\ov{\sho_\shx}$ of anti-holomorphic functions, that we denote by~$\sho_{\ov\shx}$. We~regard~$\sho_{\ov\shx}$ as an $\sho_\shx$-module by setting, for any holomorphic function~$f$ on~$\shx$, $f\cdot1:=\ov f$. Given an $\sho_\shx$-module $\shm$, the \emph{naive conjugate} module~$\ov\shm$ is defined as $\ov\shm=\sho_{\ov\shx}\otimes_{\sho_\shx}\shm$. In other words, $\ov\shm$ is equal to $\shm$ as a sheaf of $\CC$-vector spaces, and the action of $\sho_{\ov\shx}$ is that induced by the action of $\sho_\shx$. This naive conjugation functor can be regarded from $\Mod(\sho_\shx)$ to $\Mod(\sho_{\ov\shx})$ or vice-versa, and the composition of both is the identity. In~particular, the conjugate $\ov Z$ of a closed analytic subspace $Z$ of $X$ with ideal sheaf $\shi_Z$ is the closed subset $|Z|$ endowed with the sheaf $\shi_{\ov Z}:=\ov{\shi_Z}\subset\ov{\sho_\shx}$, so that~$Z$ and~$\ov Z$ have the same support.

For a complex manifold $X$, the categories of sheaves of $\CC$-vector spaces on $X$, \resp of $\pOS$-modules on $\XS$, coincide with that on $\ov X$, \resp on $\ov \XS$, since they only depend on the underlying topological space $|X|$. By~the above remark, the categories of $\RR$-constructible or $\CC$-constructible such objects also coincide. As a consequence, we have a canonical identification $\rD^\rb_{\cc}(p_X^{-1}\sho_S)=\rD^\rb_{\cc}(p_{\ov X}^{-1}\sho_S)$.
\end{notation}

The main result of this work is the following.

\begin{theorem}\label{th:main}\mbox{}
\begin{enumerate}
\item\label{th:main1}
The relative Hermitian duality functor
\[
\Conj_{X,\ov X}^S(\cbbullet):=\Rhom_{\DXS}(\cbbullet,\Db_{\XS/S})
\]
induces an equivalence
\begin{equation}\label{eq:ConjXS}
\Conj_{X,\ov X}^S:\rD^\rb_{\rhol}(\DXS)\isom\rD^\rb_{\rhol}(\DXbS)^\op
\end{equation}
such that
\begin{equation}\label{eq:ConjXSinvol}
\Conj_{\ov X,X}^S\circ\Conj_{X,\ov X}^S\simeq\id.
\end{equation}
Moreover, the relative conjugation functor $\conj_{X,\ov X}^S:=\Conj_{X,\ov X}^S\circ \bD$ induces an equivalence
\[
\conj_{X,\ov X}^S:\rD^\rb_{\rhol}(\DXS)\isom\rD^\rb_{\rhol}(\DXbS),
\]
and there is an isomorphism of functors
\begin{equation}\label{eq:ConjXSconjSol}
\pSol_{\ov X}\circ\conj_{X,\ov X}^S\simeq\pSol_X:\rD^\rb_{\rhol}(\DXS)\to\rD^\rb_{\cc}(\pOS).
\end{equation}
\item\label{th:main2}
If $\shm\!\in\!\Mod_\rhol(\DXS)$ is strict, then so is $\shh^0\Conj_{X,\ov X}^S(\shm)$ and $\shh^i\Conj_{X,\ov X}^S(\shm)=0$ for $i\neq0$.
\item\label{th:main3}
If $\shm\!\in\!\Mod_\rhol(\DXS)$ is torsion, then so is $\shh^1\Conj_{X,\ov X}^S(\shm)$ and $\shh^i\Conj_{X,\ov X}^S(\shm)=0$ for $i\neq1$.
\end{enumerate}
\end{theorem}

\begin{remark}
According to \cite{FMF18}, $\pSol$ is $t$-exact with respect to the perverse $t$-structure on $\rD^\rb_{\cc}(p^{-1}\sho_S)$ and the dual $t$-structure of the natural one in $\rD^\rb_{\rhol}(\DXS)$.
This entails that the functor $\conj_{X,\ov X}^S$ is $t$-exact with respect to the natural $t$-structures in both origin and target and the functor $\Conj_{X,\ov X}^S$ is
$t$-exact with respect to the natural one in $\rD^\rb_{\rhol}(\DXS)$ and the dual $t$-structure on $\rD^\rb_{\rhol}(\shd_{\ov{X}\times S/S})$, as seen on \eqref{th:main2} and \eqref{th:main3}.
\end{remark}

Recall (\cf\cite{MFCS1}) that a $\DXS$-module is said to be \emph{strict} if it has non nonzero $\pOS$-torsion. It is said to be \emph{torsion} if it is a torsion $\pOS$-module.

We can express the main theorem by emphasizing the functor $\RH^S$ instead of $\pSol$. For example, by the relative Riemann-Hilbert correspondence of~\hbox{\cite[Th.\,1]{FMFS19}} and \eqref{eq:ConjXSinvol}, \eqref{eq:ConjXSconjSol} can be expressed as an isomorphism of functors
\begin{equation}\label{eq:ConjXSconjRH}
\RH^S_{\ov X}\simeq \conj_{X,\ov X}^S\circ\RH^S_X:\rD^\rb_{\cc}(\pOS)\to\rD^\rb_{\rhol}(\DXbS).
\end{equation}

\subsubsection*{Ackowledgements}
We warmly thank Luisa Fiorot for useful discussions and pointing out a mistake in older Proposition \ref{P:RHD1}.

\section{Review on Kashiwara's conjugation functor}
The case where $S$ is a point (that we call the ``absolute case'') was solved by Kashiwara in \cite{Kashiwara86}, where $\Conj_{X,\ov X}^S$ is simply denoted by~$\Conj_X$. We recall here the main results in \cite{Kashiwara86} (cf. also \cite{BK86} for further study of this functor).

\pagebreak[2]
\begin{properties}\label{properties}\mbox{}
\begin{enumerate}
\item\label{enum:1}
Denoting by $\Db_X$ the left $(\DX,\DXb)$-bimodule of distributions on~$X$, the \emph{Hermitian duality functor} $\Conj_X$ is defined by
\[
\Conj_X(\shm):=\ho_{\DX}(\shm,\Db_X),
\]
endowed with the $\DX$\nobreakdash-module structure induced by that of $\Db_X$. It is an equivalence $\Mod_{\rhol}(X)\simeq\Mod_{\rhol}(\ov X)^\op$. We recall that regularity is in fact not needed, and that $\Conj_X$ extends as an equivalence $\Mod_{\hol}(X)\simeq\Mod_{\hol}(\ov X)^\op$, \cf \cite[\S II.3]{Sabbah00}, \cite[\S4.4]{Mochizuki10b}, \cite[\S12.6]{Sabbah13}, but we will not develop this direction in this article. Let us also recall that the nondegenerate associated Hermitian pairing $\shm\otimes_\CC\nobreak\Conj_X(\shm)\to\Db_X$ plays an instrumental role in the theory of mixed twistor $\shd$\nobreakdash-modules (\cf\cite{Sabbah05,Mochizuki08,Mochizuki11}).
\item\label{enum:1b}
In order to keep the same underlying manifold $X$ in the target category, we compose $\Conj_X$ with naive conjugation and obtain the equivalence $\Conj_X(\overline{\,\cbbullet\,}):\Mod_{\rhol}(X)\simeq\Mod_{\rhol}(X)^\op$.
\item\label{enum:2}
Denoting by $\bD$ the duality functor for holonomic $\DX$-modules, the \emph{conjugation functor} $\conj_X(\shm)$ is defined as $\Conj_X(\ov{\bD\shm})$. For $\shm\in\Mod_{\rhol}(X)$ (and also in $\Mod_{\hol}(X)$), $\conj_X(\shm)=\ho_{\DXb}(\ov{\bD\shm},\Db_X)$. It is thus a covariant functor, inducing an equivalence $\Mod_{\rhol}(X)\simeq\Mod_{\rhol}(X)$ (and similarly for holonomic modules). This functor satisfies
\[
\pSol(\conj_X(\shm))\simeq\ov{\pSol(\shm)}\quad\text{and}\quad\pDR(\conj_X(\shm))\simeq\ov{\pDR(\shm)},
\]
where, on the right-hand sides, we consider the naive conjugation of complexes of $\CC$-vector spaces.
\item\label{enum:4}
For $\shm$ regular holonomic (holonomic is enough, according to \eqref{enum:1}) the local cyclicity of $\Conj_X\shm$ (because it is holonomic) is equivalent to the property that $\shm$ can be locally embedded as a $\DX$-submodule of~$\Db_X$.

\item\label{enum:3}
If $u\in\Gamma(X,\Db_X)$ is a regular holonomic distribution on $X$, \ie is such that the submodule $\DX u\subset\Db_X$ is regular holonomic, then the $\DXb$-module $\DXb u\simeq\Conj_X(\DX u)\subset\Db_X$ is also regular holonomic (\cf\cite[\S3]{Kashiwara86}, \cite[Prop.\,7.4.3]{Bjork93}). In other words, the conjugate distribution~$\ov u$ also generates a regular holonomic $\DX$-submodule of $\Db_X$. (The same property holds for holonomic distributions.)
\end{enumerate}
\end{properties}

\section{Applications: relative regular holonomic distributions}
In this section we extend to the relative case some applications given in~\cite{Kashiwara86} and \cite{Bjork93}. From \cite{Bjork93} we adapt two functorial applications but a number of other applications in \loccit should be generalizable to the relative setting. We~will say for short that a morphism of $\DXS$-modules is \emph{an isomorphism up to $S$-torsion} if its kernel and cokernel are $S$-torsion $\DXS$-modules. 

Concerning the analogue of \eqref{enum:4}, the following is obtained as in \cite[Prop.\,5]{Kashiwara86}.

\begin{proposition}\label{P:RHD1}
Let $\shm$ be a strict regular holonomic $\DXS$-module and let~$\phi$ be a local section of $\Conj_{X,\ov X}^S(\shm)$, hence a $\DXS$-linear morphism from~$\shm$ to $\Db_{\XS/S}$. The following properties are equivalent:
\begin{enumerate}
\item
$\phi$ is injective,
\item
the inclusion $\shd_{\ov{X}\times S/S}\phi\subset\Conj_{X,\ov X}^S(\shm)$ is an equality up to $S$-torsion.
\end{enumerate}
\end{proposition}

\begin{proof}[Sketch of proof]
Set $\shn=\shd_{\ov{X}\times S/S}\phi$ and let $\alpha:\shn\hto\Conj_{X,\ov X}^S(\shm)$ denote the inclusion. Since $\shm$ is regular holonomic and strict, so is $\Conj_{X,\ov X}^S(\shm)$ by Theorem \ref{th:main}, hence so is $\shn$ (\cf \cite[Prop.\,3.1(iii)]{FMFS22}). We consider the exact sequence of regular holonomic $\shd_{\ov{X}\times S/S}$-modules:\vspace*{-3pt}
\[
0\to\shn\To{\alpha}\Conj_{X,\ov X}^S(\shm)\to\shn'\to0.
\]
Applying $\Conj_{\ov X,X}^S$ to it and using Theorem \ref{th:main}\eqref{eq:ConjXSinvol}, we find an exact sequence\vspace*{-3pt}
\[
0\to\shh^0\Conj_{\ov X,X}^S(\shn')\to\shm\To{\beta}\Conj_{\ov X,X}^S(\shn)\to\shh^1\Conj_{\ov X,X}^S(\shn')\to0,
\]
with $\beta=\Conj_{\ov X,X}^S(\alpha)$. The argument given in the proof of \cite[Prop.\,5]{Kashiwara86} shows that $\ker\beta=\ker\varphi$.

If $\phi$ is injective, so is $\beta$, and \eqref{th:main2} and \eqref{th:main3} in Theorem \ref{th:main} imply that $\shh^1\Conj_{\ov X,X}^S(\shn')$ is of $S$-torsion, and therefore so is $\shn'$.

Conversely, if $\shn'$ is of $S$-torsion, we deduce similarly that $\ker\beta=0$, hence~$\phi$ is injective.
\end{proof}

Let $u$ be a nonzero relative distribution, that is, a local section of $\Db_{\XS/S}$. It generates a $\DXS$-submodule\enlargethispage{\baselineskip}%
\[
\DXS\cdot u\subset\Db_{\XS/S}
\]
and a $\DXbS$-submodule $\DXbS\cdot u\subset\Db_{\XS/S}$. Let us already notice that both are \emph{strict}, \ie do not have nonzero $\pOS$-torsion elements, as follows from \eqref{enum:obviousd} in the appendix.

We denote by $\Cb_{M\times T/T}$ the sheaf of relative currents of maximal degree, that is, the sheaf of relative forms on $\XS$ of maximal degree with coefficients in $\Db_{\XS/S}$ (\cf Section \ref{subsec:currents} of the appendix).

\begin{definition}\label{DRHD}
Let $u$ be a section of $\Db_{\XS/S}$ on an open subset $\Omega$ of $\XS$. We say that $u$ is a \textit{regular holonomic relative distribution on $\Omega$} if $\shd_{\XS/S} u$ is a (strict) regular holonomic $\DXS{}_{|\Omega}$-module.
Similarly, a section $u$ of $\Cb_{M\times T/T}$ on an open subset $\Omega\subset \XS$ is a \textit{regular holonomic relative current of maximal degree on $\Omega$} if $u \shd_{\XS/S} $ is a regular holonomic right $\DXS{}_{|\Omega}$-module.
\end{definition}

We obtain the analogue of Property \eqref{enum:3} above in the relative setting.
\begin{corollary}\label{PRHD}
Let $u$ be a section of $\Db_{\XS/S}$ on an open subset $\Omega$ of $\XS$. Then the following conditions are equivalent:
\begin{enumerate}
\item\label{PRHD1}
$\DXS u$ is regular holonomic on $\Omega$,
\item\label{PRHD2}
$\DXbS u$ is regular holonomic on $\ov{\Omega}$.
\end{enumerate}
Furthermore, in such a case, $\DXbS u\simeq\Conj_{X,\ov X}^S(\DXS u)$ up to $S$-torsion.
\end{corollary}

\begin{proof}
Assume that \eqref{PRHD1} holds for $u$ and let $\shj\subset\DXS$ be the left ideal consisting of operators $P$ such that $Pu=0$ in $\Db_{\XS/S}$. Let us set $\shm=\DXS/\shj$. By definition, the $\DXS$-linear morphism
\[
\DXS\to \shd_{\XS/S} u,\qquad 1\mto u,
\]
induces an isomorphism $\shm\isom\shd_{\XS/S} u$, that we regard as an injective $\DXS$-linear morphism $\phi:\shm\hto\Db_{\XS/S}$, in particular a local section of $\ho_{\DXS}(\shm, \Db_{\XS/S})$. Thus the coherent $\DXbS$-submodule $\DXbS \phi$ of $\Conj_{X,\ov X}^S(\shm)$ is regular holonomic.  
Since $\phi$ is injective, Proposition~\ref{P:RHD1} implies that $\Conj_{X,\ov X}^S(\shm)/\DXbS\phi$, which is also regular holonomic, is of $S$-torsion. By construction, $\DXbS\phi$ is nothing but $\DXbS u$. This proves~\eqref{PRHD2} and the supplementary assertion of the proposition. Changing $u$ to $\ov u$ yields the converse.
\end{proof}

\begin{remark}
Embedding locally a regular holonomic $\DXS$-module $\shm$ in $\Db_{\XS/S}$ by a morphism $\phi$ can only occur if $\shm$ is strict, and even then, as we saw in Proposition \ref {P:RHD1}, this is not sufficient to conclude that $\Conj_{X,\ov X}^S\shm$ is locally cyclic as a $\DXbS$-module (an analogue of Property \ref{properties}\eqref{enum:4}); this holds only up to $S$-torsion, that is, away from the support of the cokernel of $\DXbS\phi\hto \Conj_{X,\ov X}^S\shm$. On the other hand, again in contrast with the absolute case, we do not know whether the local cyclicity property holds for any strict regular holonomic $\DXS$ or $\DXbS$-module. We consider examples in Section \ref{sec:3}. Nevertheless, Corollary \ref{PRHD} implies that, for $\shm$ regular holonomic and strict, the conjunction of~$\shm$ and $\Conj_{X,\ov X}^S\shm$ being locally cyclic implies that both are, up to $S$-torsion, locally generated (over $\DXS$ \resp $\DXbS$) by the same distribution~$u$.
\end{remark}

By definition, a strict holonomic $\DXS$-module $\shm$ is regular if and only if $i^*_s\shm$ is a regular holonomic $\shd_X$-module. This translates as follows for relative distributions.

\begin{proposition}\label{serie}
Let $\Omega'$ be an open subset in $\C^n$, $V$ an open disc centered at the origin in $\C$ and $\Omega=\Omega'\times V$. Let $u\in \Gamma (\Omega; \Db_{\XS/S})$ and let
\[
u=\sum_{m\geq 0}u_ms^m
\]
be the expansion of $u$ provided by Corollary \ref{cor:expansion}. If $u$ is regular holonomic, then each $u_m$ is regular holonomic. Conversely, any finite sum $\sum_{m= 0}^Nu_ms^m$, with $u_m\in\Gamma(\Omega';\Db_X)$ regular holonomic for each $m$, is regular holonomic.
\end{proposition}

\begin{proof}
We first notice that, for $v\in \Gamma (\Omega; \Db_{\XS/S})$ and $f\in\sho(V)$ nonzero, if $f(s)v$ is regular holonomic, then so is $v$, since for $P\in\Gamma (\Omega; \DXS)$, the relation $f(s)Pv=0$ in $\Gamma (\Omega; \Db_{\XS/S})$ is equivalent to $Pv=0$. We can therefore assume in the proof that $u_0\neq0$. If $\shi_u\subset\Gamma (\Omega; \Db_{\XS/S})$ denotes the ideal of operators satisfying $Pu=0$ and if we write for such an operator $P=P_0+sP'$ with $P_0\in\Gamma(\Omega',\shd_X)$, we have $i^*_0\shi_u=\{P_0\mid P\in\shi_u\}$. The relation $Pu=0$ implies $P_0u_0=0$, so $u_0=i^*_0u$ is annihilated by $i^*_0\shi_u$, hence is regular holonomic. It~follows that $u-u_0$ is also relatively regular holonomic, and processing that way we obtain the first assertion.

For the second assertion, we note that $\shd_{\XS/S}u$ is a coherent submodule of $\shd_{\XS/S}\otimes_{\shd_X}(\bigoplus_m\shd_Xu_m)$ and the assertion follows since $\Mod_{\rhol}(\DXS)$ is stable by sub-quotients in $\Mod_{\coh}(\DXS)$ (\cf \cite[Prop.\,3.1(iii)]{FMFS22}).
\end{proof}

We also prove, by mimicking the proof given by Björk \cite[\S\S VII.4 \& VII.5]{Bjork93}, that regular holonomicity of relative distributions (and currents) is preserved under pushforward with proper support and non-characteristic pullback.

We say that an embedding \hbox{$i_Y:Y\times S\hto \XS$} of a closed submanifold~$Y$ of $X$ is non-characteristic for a coherent $\DXS$-module~$\shm$~if
\[
\Char(\shm)\cap T^*_Y \XS\subset T^*_X \XS.
\]
In such a case, $\Di_Y^*\shm=\shh^0\Di_Y^*\shm$ is $\DYS$-coherent. Moreover, according to \cite[Th.\,2]{FMFS19} it follows that, if $ \shm$ is regular holonomic, then so is $\shh^0\Di_Y^*\shm$.

Let $M$ be a real analytic manifold, let $u$ be a section of $\Db_{M\times S/S}$ on an open subset $\Omega\subset M\times S$. We say that $u$ is regular holonomic on~$\Omega$~if, for each $(m, s)\in\Omega$, there exists a neighborhood $V\times W\subset \Omega$ of $(m,s)$, a complexification $X$ of $V$, a coherent left ideal $\shi$ of $\DXS$ defined in a neighborhood $\Omega':=V'\times W\subset V\times W$ of $(m,s)$ in $\XS$ such that \hbox{$\shi|_{\Omega'\cap M\times S} \cdot u|_{\Omega'\cap M\times S}=0$} and $\DXS/\shi$ is regular holonomic. In that situation, we denote $\DXS u$ instead of $\DXS/\shi$.

We finally recall that the analytical wave front set of a distribution $u$ on a real manifold $Z$, noted by $WF_A(u)$, is an $\R_+$- conic closed subset of $T^*Z$ which can be defined as the support of the Sato's microfunction defined by~$u$.

\begin{proposition}\mbox{}\label{P10}
\begin{enumerate}
\item\label{P10a}
Let $u$ be a relative regular holonomic current of maximal degree on $\XS$ such that $(f\times\id)|_{\Supp u}$ is proper. Then $\int_fu$ is a regular holonomic relative current on $\YS$.

\item\label{P10b}
Let $N\subset M$ be real analytic manifolds and let $i_Y:Y\subset X$ be their respective germs of complexifications (in particular, $Y\cap M=N)$. Let~$\Omega$ be an open subset of $M\times S$ and let $u\in \Gamma(\Omega,\Db_{M\times S/S})$ be such that $\shm:=\DXS u$ is regular holonomic. Assume that $Y\times S$ is non-characteristic for $\shm$, so that $\Di_Y^*\shm=\shh^0\Di_Y^*\shm$ is regular holonomic. Then the restriction $u_{|N\times S}$ is a relative distribution on $(\Omega\cap N)\times S$ such that $\DYS\cdot u_{|N\times S}\simeq\shh^0\Di_Y^*\shm|_{N\times S}$.
\end{enumerate}
\end{proposition}

\pagebreak[2]
\begin{proof}\mbox{}
\begin{enumerate}
\item
This part is completely similar to that of \cite[Th.7.4.11]{Bjork93}. One shows that the integral of $u$ along $f\times\id$, as defined in Section \ref{subsec:currents}, is a section of the $0$th pushforward of the $\DXS$-module $\DXS\cdot u$. This uses Proposition \ref{prop:Cacyclic}. Note that, since $\DXS\cdot u$ is globally generated, it is $f$-good. It~follows then from the stability of regular holonomic $\DXS$-modules by $f$-good pushforward (\cf\cite[Cor.\,2.4]{MFCS2}) that the latter is a regular holonomic $\DYS$-module. We~conclude by using the property that the category $\Mod_{\rhol}(\DXS)$ is stable by sub-quotients in $\Mod_{\coh}(\DXS)$, as already mentioned in the proof of Proposition~\ref{serie}.

\item
Let us prove \ref{P10}\eqref{P10b}. We first remark that $T_N^*M=T^*_MX\cap T^*_YX$. Hence, by the assumption, according to a well-known result by M.\,Sato, the analytical wave front set of $u$ seen as a distribution on $M\times S_{\R}$, $WF_A(u)$, satisfies
\[
WF_A(u)\cap (T^*_N M\times S)\subset T^*_M M\times S
\]
The proof then proceeds like in \cite[7.5.4]{Bjork93} in view of the fact that~$\partial_{\ov{s}}$ is a continuous linear operator on $\Db_{M\times S_{\R}}$ commuting with the restriction to $N\times S_{\R}$ of $C^{\infty}$-functions on open subsets of $M\times S_{\R}$. Therefore, by construction, $u_{|N\times S}$ is a relative distribution on $(N\cap \Omega)\times S$ such that $\DYS\cdot u_{|N\times S}\simeq\shh^0\Di^*\shm|_{N\times S}$.\qedhere
\end{enumerate}
\end{proof}

In the absolute case, Andronikof in \cite[Th.\,1.2]{A} and Björk in \cite[Th.\,8.11.8]{Bjork93} proved that, if $u$ is a regular holonomic distribution, then the characteristic variety of the regular $\shd_X$-module $\shd_X u$ satisfies
\[
WF_A(u)=\Char(\shd_X u).
\]

For relative regular holonomic distributions, the inclusion $\subset$ is obvious:

\begin{proposition}\label{WFA}
Given a regular holonomic relative distribution $u$, we have $WF_A(u)\subset\Char(\DXS u)$ (where $S$ is identified to $T^*_SS$).
\end{proposition}

\begin{proof}
Regarding $u$ as a distribution solution of the coherent $\shd_{\XS}$-module $\shd_{\XS}\otimes_{\DXS}\DXS u$, we have $WF_A(u)\subset \Lambda\times T^*S$. On the other hand, since $u$ satisfies $\partial_{\ov{s}}u=0$, it follows that
$WF_A(u)\subset \Lambda\times S$.
\end{proof}

In order to prove the converse inclusion in Proposition \ref{WFA}, one should develop a relative variant of the microlocal techniques used in \cite{A} which is out of the scope of this work.

\remark
 In the conditions of Corollary \ref{cor:expansion}, if $\phi=\sum_{i=0}^L \phi_is^i$ then $\phi$ is regular holonomic in the absolute sense, that is, as a section of $\Db_{X\times S}$. Indeed each term $\phi_i s^i$ is regular holonomic in the absolute sense since it satisfies the ideal generated in $\shd_{X\times S}$ by the the annihilator ideal of $\phi_i$ in $\shd_X$ and by $\partial_s^{i+1}$.


\begin{example}
Let us assume $X=\C=S$, and let us consider the relative regular holonomic distribution $\phi=x+\delta(0)s$. Let $\DXS \phi=\DXS \shj$.  If $P\in\shj$ that is, $Px=-P\delta(0)s$, then clearly $Px=P\delta(0)=0$ and conversely. Hence $\Char(\DXS\phi)=(T^*_{\{0\}}X\cup T^*_XX)\times S=WF_A(\phi)$. 
\end{example}

\begin{example} For $X=S=\C$, $\phi(z,\ov{z}, s)=e^{s\ov{z}}$ and $|z|^s$ are examples of relative regular holonomic distributions which are not regular holonomic in the absolute sense (they are not even holonomic).
\end{example}

\section{Proof of the main theorem}

We first reduce the proof of Theorem \ref{th:main} to the following three statements.

\begin{enumerate}\renewcommand{\theenumi}{\roman{enumi}}
\item\label{enum:i}
$\Conj_{X,\ov X}^S$ sends $\rD^\rb_{\rhol}(\DXS)^\op$ to $\rD^\rb_{\rhol}(\DXbS)$.
\item\label{enum:ii}
For any $s_o\in S$, $Li_{s_o}^*\circ\Conj_{X,\ov X}^S\simeq\Conj_{X}$.
\item\label{enum:iii}
\eqref{eq:ConjXSconjSol} holds.
\end{enumerate}

\begin{proof}[Proof of Theorem \ref{th:main}\eqref{th:main1} and \eqref{th:main2} assuming \eqref{enum:i}, \eqref{enum:ii} and \eqref{enum:iii}]
Since the family of functors $(Li_{s_o}^*)_{s_o\in S}$ is conservative on $\rD^\rb_{\hol}(\DXS)$, we deduce from \eqref{enum:i}, \eqref{enum:ii} and \cite{Kashiwara86} that
\eqref{eq:ConjXS} and \eqref{eq:ConjXSinvol} in (\ref{th:main1})
hold. Then \eqref{enum:iii} is enough to conclude the proof of \eqref{th:main1}.

For \eqref{th:main2}, we use that $Li_{s_o}^*\shm\in\Mod_\rhol(\DX)$ if $\shm$ is strict, and the assertion follows from \cite[Th.\,2(i)]{Kashiwara86} and \eqref{enum:ii}, together with \cite[Prop.\,1.9]{MFCS2}.
\end{proof}

\begin{proof}[Proof of \eqref{enum:ii}]
We note that, by Proposition \ref{memento} in the appendix, we have $Li^{*}_{s_o}\Db_{\XS/S}\simeq \Db_X$ for any $s_o\in S$. We~conclude by applying \cite[(A.10)]{Kashiwara03}.
\end{proof}

\begin{proof}[Proof of \eqref{enum:iii} assuming \eqref{enum:i}]
We have
\begin{align*}
\pDR_{\ov X} (\Conj_{X,\ov X}^S(\shm))&=\Rhom_{\shd_{\XbS/S}}(\sho_{\XbS}, \Rhom_{\DXS}(\shm, \Db_{\XS/S}))[d_X]
\\
&\overset{(*)}\simeq \Rhom_{\DXS}(\shm,\Rhom_{\shd_{\XbS/S}}(\sho_{\XbS}, \Db_{\XS/S}))[d_X]
\\
&\overset{(**)}\simeq \pSol_{X}(\shm),
\end{align*}
where $(*)$ follows from a standard argument (\cf \cite[p.\,241]{Kashiwara03}) and $(**)$ follows from the relative Dolbeault-Grothendieck lemma $\DR_{\ov X}(\Db_{\XS/S})\simeq \sho_{\XS}$ (\cf\cite[\S3.1]{SS1}). Since $\shm$ and $\Conj_{X,\ov X}^S(\shm)$ have holonomic cohomologies (according to \eqref{enum:i} for the latter), we can apply \cite[(5) \& Cor.\,3.9]{MFCS1} to obtain
\begin{align*}
\pSol_{\ov X}(\conj_{X,\ov X}^S(\shm))&\simeq\bD\pDR_{\ov X} (\Conj_{X,\ov X}^S\bD\shm)\\
&\simeq\bD\pSol_{X}(\bD\shm)\simeq\pSol_{X}(\shm).\qedhere
\end{align*}
\end{proof}

The proof of \eqref{enum:i} will follow the same strategy as for that of \cite[Th.\,3]{FMFS19}. We consider the statement $P_X$ defined, for any complex manifold $X$ and any $\shm\in\rD^\rb_{\rhol}(\DXS)$, by\vspace*{-3pt}
\[
\tag*{$P_X(\shm)$:}\Conj_{X,\ov X}^S(\shm)\in\rD^\rb_{\rhol}(\DXbS),
\]
in other words, \eqref{enum:i} holds for $\shm$. Properties (a)--(d) of \cite[Lem.\,3.6]{FMFS19} clearly hold for $P_X$. We need to check Property (f) of \loccit, that is, the truth of $P_X(\shm)$ when $\shm$ is torsion.

The following result also proves \eqref{th:main3} of Theorem \ref{th:main}:
\begin{lemma}\label{lem:torsion}
Assume that $\shm\in\Mod_\rhol(\DXS)$ is torsion. Then $P_X(\shm)$ is true. Moreover, $\Conj_{X,\ov X}^S(\shm)[1]\simeq\shh^1\Conj_{X,\ov X}^S(\shm)$.
\end{lemma}

\begin{proof}
The property is local with respect to $X$ and $S$, and we can assume that $S=\CC$ with coordinate $s$ and that $s^N\shm=0$ for some $N\geq1$. By a standard extension argument, and due to Property (c) of \cite[Lem.\,3.6]{FMFS19}, we can assume that $s\shm=0$.

Then, by division by $s$, we can write $\shm\simeq q_X^{*}\shm'\otimes_{p^{-1}_X\sho_S}p^{-1}_X(\sho_S/s\sho_S)$ where $q_X$ denotes the projection $\XS\to X$, for some regular holonomic $\shd_X$-module $\shm'$.
Hence\vspace*{-3pt}
\begin{align*}
\Conj_{X,\ov X}^S(&\shm)\simeq \Rhom_{\DXS}(q_X^{*}\shm'\otimes_{p^{-1}_X\sho_S}p^{-1}_X(\sho_S/s\sho_S), \Db_{\XS/S})\\
&\simeq \Rhom_{\DXS}(q_X^{*}\shm', \Rhom_{p^{-1}_X\sho_S}(p^{-1}_X(\sho_S/s\sho_S), \Db_{\XS/S}))\\
&\simeq \Rhom_{\DXS}(q_X^{*}\shm', \Db_{\XS/S}/s\Db_{\XS/S})[-1]\\
&\simeq\Rhom_{\DXS/s\DXS}(q_X^{*}\shm'/sq_X^{*}\shm', \Db_{\XS/S}/s\Db_{\XS/S})[-1]\\
&\simeq \Rhom_{\shd_X}(\shm', \Db_{X})[-1],
\end{align*}
which has regular holonomic cohomology over $\shd_{\bar{X}}$ concentrated in degree 1 according to \cite [Th.\,2]{Kashiwara86}.
\end{proof}
\begin{proof}[Proof of \eqref{enum:i}]
We consider the analogue of the statement $P$ for right $\DXS$-modules:\vspace*{-3pt}
\[
\tag*{$P_X(\shn)$:}\shn\otimes^L_{\DXS} \Db_{\XS/S}\in \rD^\rb_{\rhol}(\shd_{\XbS/S}).
\]
We first remark that, by biduality, \eqref{enum:i} is equivalent to the following condition, working with right $\DXS$-modules $\shn$:
\begin{enumerate}\renewcommand{\theenumi}{\roman{enumi}$'$}
\item\label{enum:i'}
For all $\shn\in\rD^\rb_{\rhol}(\DXS)$, $P_X(\shn)$ holds true.\end{enumerate}
In what follows we will prove \eqref{enum:i'}. We argue in a way similar to that of \hbox{\cite[End of the proof of Th.\,3]{FMFS19}}, and we only emphasize the modifications of the argument.

\subsubsection*{Step 1. Proof of \eqref{enum:i'} in the torsion case}
This follows from Lemma \ref{lem:torsion}.

\subsubsection*{Step 2}
In a way identical to that used in \cite[End of the proof of Th.\,3]{FMFS19}, we reduce to proving $P_X(\shn)$ for those regular holonomic $\DXS$-modules $\shn$ which are strict and satisfy
\begin{enumerate}
\item
the support $Z\subset X$ of $\shn$ is pure dimensional of dimension $k$,
\item
there exists a closed hypersurface $Y\subset X$ such that
\begin{itemize}
\item
$Y\cap Z$ contains the singular locus of $Z$,
\item
$\shn\simeq\shn(*(Y\times S))$.
\end{itemize}
\end{enumerate}

\subsubsection*{Step 3. Reduction to the case of D-type}
Let us recall the definition of D-type (\cf \cite[Def.\,2.10]{MFCS2}).

\begin{definition}\label{def:Dtype}
Let $Y$ be a normal crossing hypersurface in $X$. We recall that a regular holonomic $\DXS$-module $\shn$ is said to be of D-type along $Y$ if it is strict, if its relative characteristic variety is contained in $(T^*X_{|Y}\cup T^*_XX)\times S$ and if $\shn=\shn(*(Y\times S))$. Equivalently (\cf\cite[Prop.\,2.11]{MFCS2}), there exists a locally free $p_{X\moins Y}^{-1}\sho_S$-module $F$ such that the left $\DXS$-module associated with~$\shn$ is equal to the extension by moderate growth of $F\otimes_{\pOS}\sho_{(X\moins Y)\times S}$.
\end{definition}

In the situation of Step 2, there exists a commutative diagram
\[
\begin{array}{c}
\xymatrix{
& X'\setminus Y'
\ar@{^{ (}->}[r]^(.65){j'} \ar[ld]_{\pi_{Z^*}}
&
X'\ar[d]^{\pi} \\
{Z^*}
\ar@{^{ (}->}[r]^-{i} &X\moins Y \ar@{^{ (}->}[r]^-{j} & X
}
\end{array}
\]
where $\pi$ is a projective morphism, $X'$ is smooth, $Y'$ is a normal crossing hypersurface, $Z^*:=Z\setminus Y$ and $\pi_Z^*$ is biholomorphic. Let
\[
\Dpi^*\shn=\shn\otimes^L_{\DXS}\shd_{(X\leftarrow X')/S}
\]
be the pullback of $\shn$. Then $\shn':=\Dpi^*\shn[d_{X'}-d_X]$ is concentrated in degree zero and is of D-type along $Y'$ and we have $\shn\simeq\Dpi_*\shn'$. We assume that $P_{X'}(\shn')$ holds and we prove that $P_X(\shn)$ holds.

According to \cite[Cor.\,2.4]{MFCS2}, the pushforward $\Dpi_*(\shn'\otimes^L_{\shd_{\XpS/S}}\Db_{\XpS/S})$ is an object of $\rD^\rb_\rhol(\DXbS)$. Then we can argue as in \cite[p.\,206]{Kashiwara86} to conclude that $P_X(\shn)$ holds.

\subsubsection*{Step 4. Proof of \eqref{enum:i'} in the case of D-type}

Let $Y$ be a normal crossing divisor in $X$ and assume that $\shn$ is a right $\DXS$-module of D-type along~$Y$. In particular, $\shn$ is strict and $\shn\simeq \shn(*(Y\times S))$. We are thus reduced to proving that
\[
\shn\overset{L}{\otimes}_{\DXS}\Db_{\XS/S}(*(Y\times S))
\]
has regular holonomic cohomology as a $\DXbS$-module. The main point is to prove coherence and holonomicity.

Since the proof is local, we may assume that, in a neighbourhood of $(x_0, s_o)\in \XS$, $S=\C$ with a coordinate $s$ vanishing at $s_o$, $X=\C^d$ with coordinates $(x_1,\dots, x_d)$ vanishing at $x_0$, $Y=\{f(x)=0\}$ with $f=x_1\cdots x_\ell$ and, by strictness (\cf\cite[Prop.\,2.11 \& proof of Cor.\,2.8]{MFCS2}), after a convenient ramification $\rho:S'\to S$ at $s_o$ of finite order, the pullback $\rho^*\shn$
\begin{enumerate}
\item\label{assumption:i}
has a finite filtration whose successive quotients take the form $\shd_{\XS'/S'}/(P_1,\dots,P_d)$, with
\[
P_1=x_1\partial_1-\alpha_1(s),\dots, P_\ell=x_\ell\partial_\ell-\alpha_\ell(s),\ P_{\ell+1}=\partial_{\ell+1},\dots, P_d=\partial_d,
\]
for some holomorphic functions on $S'$ such that $\alpha_j(s_o)\notin\Z_{-}$ (since~$\shn$ is localized along $Y\times S$ as a right $\DXS$-module).
\end{enumerate}
We prove the statement for $\shn$ assuming it for $\rho^*\shn$. Let us note that
\[
\rho^*(\shn\otimes^L_{\DXS}\Db_{\XS/S})\simeq \rho^*\shn\otimes^L_{\shd_{\XS'/S'}}\Db_{\XS'/S'}
\]
functorially on $\shn$. In fact, this holds for any $\DXS$-coherent module $\shn$: to see that, we can locally reduce to the case were $\shn=\DXS$, that is, to prove that $\rho^*\Db_{\XS/S}=\Db_{\XS'/S'}$, which is true by Corollary~\ref{Cram}. Since $\shn\simeq\shn(*(Y\times S))$, we also have
\begin{equation}\label{E9}
\rho^*\shn\otimes^L_{\shd_{\XS'/S'}}\Db_{\XS'/S'}\simeq\rho^*\shn\otimes^L_{\shd_{\XS'/S'}}\Db_{\XS'/S'}(*(Y\times S')).
\end{equation}

Since $\shn=\shn(*(Y\times S))$, the localized pushforward (in the sense of $\shd$\nobreakdash-modules) $(\Drho_*\rho^*\shn)(*(Y\times S))$ is equal to the sheaf pushforward $\rho_*\rho^*\shn$, and by Corollary \ref{Cram} and the projection formula, we have functorial isomorphisms of $\DXbS(*(\ov Y\times S))$-modules
\begin{equation}\label{E10}
\begin{split}
\Drho_*(\rho^*\shn\otimes^L_{\shd_{\XS'/S'}}&\Db_{\XS'/S'}(* Y\times S'))\\
&\simeq\rho_*(\rho^*\shn\otimes^L_{\shd_{\XS'/S'}}\Db_{\XS'/S'}(* Y\times S'))\\
&\simeq \rho_*\rho^*\shn\otimes^L_{\DXS}\Db_{\XS/S}(* Y\times S).
\end{split}
\end{equation}
Hence $\shn\otimes_{\DXS}^L\Db_{\XS/S}$, being a direct summand of the holonomic $\DXbS$-module $\rho_*\rho^*\shn\otimes_{\DXS}\Db_{\XS/S}(*(Y\times S))$, is holonomic.

To check the regularity of $\shn\otimes_{\DXS}^L\Db_{\XS/S}$, we apply $Li^*_{s_o}$ for each \hbox{$s_o\in S$} and recall (\cite[Prop.\,2.1]{MFCS1}) that the result is functorially isomorphic~to
\[
Li^*_s\shn\otimes^L_{\shd_X} Li^ *_s\Db_{\XS/S}
\]
hence, according to Proposition \ref{memento} and to \cite[Th.\,1]{Kashiwara86}, $\shn\otimes_{\DXS}^L\Db_{\XS/S}$ is regular holonomic.

It remains therefore to prove the statement for $\shn$ assuming that it satisfies~\eqref{assumption:i}. It is then enough to prove the statement for each successive quotient of the corresponding filtrations, so we assume that $\shn$ is one such quotient. In such a case,~$\shn$ has a resolution by free $\DXSp(*(Y\times S'))$-modules obtained by taking the simple complex associated with the $d$-cube complex having vertices equal to $\DXSp(*(Y\times S'))$ and arrows in the $i$th direction equal to $P_i$. It follows that the later complex is isomorphic to the simple complex associated with the $d$-cube complex having vertices equal to $\Db_{\XS'/S'}(*(Y\times S'))$ and arrows in the $i$th direction equal to the left action of $P_i$ on $\Db_{\XS'/S'}(*(Y\times S'))$. Note also that $\Db_{\XS'/S'}(*(Y\times S'))=\Db_{\XS'/S'}(*(\ov Y\times S'))$.

The function $(x,s)\mto|x|^{2\alpha(s)}:=|x_1|^{2\alpha_1(s)}|x_2|^{2\alpha_2(s)}\cdots|x_\ell|^{2\alpha_\ell(s)}$ is real analytic on $(X\setminus Y)\times S'$, holomorphic in $s$, and defines, by multiplication, an $(\sho_{\XS'},\sho_{\XbS'})$-linear isomorphism from $\Db_{\XS'/S'}(*(Y\times S'))$ to itself. \hbox{Regarding} $P_i$ as $|x|^{2\alpha_i(s)}\circ(x_i\partial_i)\circ|x|^{-2\alpha_i(s)}$ for $i=1,\dots,\ell$, \eqref{E9} is isomorphic to the simple complex associated with the $d$-cube complex having vertices equal to $\Db_{\XS'/S'}(*(Y\times S'))$ and arrows in the $i$th direction equal to the left action of~$Q_i$ on $\Db_{\XS'/S'}(*(Y\times S'))$, with $Q_1=x_1\partial_1,\dots, Q_\ell=x_\ell\partial_\ell,\, Q_{\ell+1}=\partial_{\ell+1},\dots, Q_d=\partial_d$.

By the relative Dolbeault-Grothendieck lemma, the latter complex has cohomology in degree zero only. It follows that so does \eqref{E9}, and its zeroth cohomology is isomorphic to $\sho_{\XbS'}(*(\ov{Y}\times S'))$ endowed with its usual connection twisted by $|x|^{-2\alpha(s)}$. In other words, it is of D-type on $(\ov X,\ov Y)\times S$ and, is thus holonomic.
\end{proof}

\section{Some examples}\label{sec:3}
\begin{proposition}\label{DT}
Let us assume that $d_X>0$. Any regular holonomic $\DXS$-module $\shm$ of D-type along a normal crossing divisor $Y\subset X$ (Definition \ref{def:Dtype}) locally admits an injective $\DXS$-linear morphism taking values in $\Db_{\XS/S}$. Furthermore,
$\Conj_{X,\ov X}^S(\shm)$ and $\bD\shm$ are locally cyclic up to $S$\nobreakdash-torsion.
\end{proposition}

\begin{proof}
In the proof of Lemma 4.2 of \cite{MFCS2} it is shown that a regular holonomic $\DXS$-module $\shm$ is of D-type along $Y$ if and only if $\pSol\shm$ takes the form $j_!F[d_ X]$ for some $S$-locally constant sheaf of locally free $\pOS$-modules $F$ on $X\moins Y$, where $j:X\moins Y\hto X$ denotes the inclusion, and then $\shm\simeq\RH^S(j_!F[d_X])$, according to the Riemann-Hilbert correspondence of \loccit Then, by Theorem \ref{th:main}\eqref{eq:ConjXSconjSol}, $\conj_{X,\ov X}^S(\shm)$ is also of D-type on~$\XbS$. Therefore, if the first statement of the proposition is proved, we conclude from Proposition~\ref{P:RHD1} that $\Conj_{X,\ov X}^S(\shm)$ as well as $\Conj_{\ov X,X}^S(\conj_{X,\ov X}^S(\shm))\simeq\bD\shm$ are locally cyclic up to $S$\nobreakdash-torsion, according to Theorem \ref{th:main}\eqref{eq:ConjXSinvol}.

Let us prove the first statement. Since it is local on $X$ and $S$, we can choose local coordinates $x_1,\dots,x_d$ on $X$ such that $Y=\{x_1\cdots x_\ell=0\}$ and, according to \cite[Prop.\,2.11]{MFCS2} and the first part of the proof of Theorem 2.6 in \loccit, we can assume that $\shm$ is the free $\sho_{\XS}(*(Y\times S))$-module $\sho_{\XS}(*(Y\times S))^k$ endowed with the relative connection $\nabla$ with matrix $\sum_{i=1}^\ell A_i(s)\rd x_i$ in the canonical $\sho_{\XS}(*(Y\times S))$-basis $(e_1,\dots,e_k)$ of $\sho_{\XS}(*(Y\times S))^k$ (where $A_i(s)$ are matrices depending holomorphically on~$s$).
\subsubsection*{Step 1}
We first embed $\sho_{\XS}(*(Y\times S))^k$ with its standard $\DXS$-module structure in $\Db_{\XS/S}$. Let us set $f(x)\!=\!x_1\cdots x_\ell$ and $u_j\!=\!\ov x_1^j/f$ (\hbox{$j\!=\!1,\dots,k$}). Then $u_j$ is a section of $\Db_{\XS/S}$ and $1/f\mto u_j$ induces an isomorphism $\sho_{\XS}(*(Y\times S))\simeq\DXS\cdot u_j$. The surjective $\DXS$-linear morphism\vspace*{-5pt}
\begin{equation}\label{eq:embedO}
\begin{split}
\sho_{\XS}(*(Y\times S))^k&\to\sum_{j=1}^k\DXS u_j\subset\Db_{\XS/S}\\[-5pt]
e_j&\mto u_j
\end{split}
\end{equation}
is an isomorphism: indeed, it suffices to show that its restriction to \hbox{$(X\moins Y)\times S$} is injective. The target of the latter is equal to $\sum_j\sho_{(X\moins Y)\times S}u_j$; if $\sum_jg_ju_j=\nobreak0$ is a relation with $g_j$ holomorphic, then applying successively $\partial_{\ov x_1}^k,\partial_{\ov x_1}^{k-1},\dots$ to this relation gives successively $g_k=0$, $g_{k-1}=0,\dots$
\subsubsection*{Step 2}
Locally on $\XS$ there exists $N_i>0$ such that each entry of the matrix $|x_i|^{2A_i(s)}$ becomes locally bounded after being multiplied by $\ov x_i^{N_i}$. We~then twist the $\sho_{\XS}(*(Y\times S))$-basis $(u_1,\dots,u_k)$ by setting
\[
(v_1,\dots,v_k)=(u_1,\dots,u_k)\cdot\prod_{i=1}^\ell\ov x_i^{N_i}|x_i|^{2A_i}.
\]
The $\sho_{\XS}$-linear morphism
\begin{align*}
\sho_{\XS}(*(Y\times S))^k&\to\Db_{\XS/S}\\
e_j&\mto v_j
\end{align*}
is thus $\DXS$-linear when using the $\DXS$-structure induced by $\nabla$ on the left-hand side. Since it is obtained by applying to \eqref{eq:embedO} a matrix which is invertible on $(X\moins Y)\times S$, it is also injective.
\end{proof}

\begin{example}\label{f}
We give an example, in general not of D-type, for which both~$\shm$ and $\Conj^S_{X,\ov X}(\shm)$ are strict, regular and locally cyclic up to $S$\nobreakdash-torsion.

Let $X$ be a relatively compact open subset of $\CC^d$ and let $u$ be a regular holonomic distribution in the sense of \cite{Kashiwara86} on some open neighbourhood of the closure of $X$. Let $f$ be a holomorphic function on $X$. Let $M=\shd_X[s]\cdot \text{``}f^s u\text{''}$ be the $\shd_X[s]$-module introduced in \cite[2.2]{Kashiwara78}. Recall that $M$ is the quotient of $\shd_X[s]$ by the left ideal $\shj$ consisting of the operators $P(s)$ such that the operator $f^{\ord(P)-s}P(s)f^s\in\shd_X[s]$ satisfies $(f^{\ord(P)-s}P(s)f^s)u=0$. Let $\shm:=\DXC\cdot \text{``}f^s u\text{''}$ be the analytification of $M$, \ie $\shm=\sho_{\XC}\otimes_{\sho_X[s]}M$. Then $\shm$ is holonomic (\cf \cite[Prop.\,13]{Maisonobe16}), strict and regular since $\shd_X\cdot u$ is regular.

Let $p$ denote the finite order of $u$ on the closure of $X$. Then $|f|^{2s}u$ is a relative distribution on $X\times \{s\in\C\mid\reel s>p\}$. By using a Bernstein relation for $\text{``}f^s u\text{''}$, one finds by a standard procedure a finite set $A\subset\CC$ (the roots of a Bernstein equation for $\text{``}f^s u\text{''}$ and $\text{``}f^s \ov u\text{''}$) such that $|f|^{2s}u$ extends as a relative distribution $v$ on $\XS$, with $S:=\CC\moins(A-\NN)$.

Let $P(s)$ be a differential operator in $\shj$ of order $\ord(P)$. Assume that $\reel s>\max\{p,\ord(P)\}$. Then, since $P(s)$ commutes with $\ov f^s$, we have
\[
P(s)v=P(s)|f|^{2s}u=\ov f^s\cdot f^{s-\ord(P)}(f^{\ord(P)-s}P(s)f^s)u=0.
\]
This identity extends to $s\in S$, according to \eqref{enum:obviousc} in the appendix. There is thus a well-defined $\DXS$-linear injective morphism $\phi:\shm_{|\XS}\hto\Db_{\XS/S}$:
\begin{align*}
\phi:\shm_{|\XS}&\isom\DXS\cdot v\subset\Db_{\XS/S}\\
\text{``}f^s u\text{''}&\mto v.
\end{align*}
Hence $\phi$ is a generator of $\Conj^S_{X,\ov X}(\shm)_{|\XS}$ up to $S$\nobreakdash-torsion. From Corollary~\ref{PRHD} we also deduce an isomorphism up to $S$\nobreakdash-torsion:
\[
\Conj^S_{X,\ov X}(\shm_{|\XS})\simeq\shd_{\XbS}\cdot v\simeq\shd_{\XbS}\cdot\text{``}\ov f^s u\text{''}.
\]

If $\alpha:S'\to\CC$ is any holomorphic function on a Riemann surface $S'$, then the pullback $\alpha^*\shm=\DXSp\cdot\text{``}f^{\alpha(s')} u\text{''}$ with respect to $S$ is regular holonomic and $\Conj_{X,\ov X}^{S'}\shm=\alpha^*\Conj_{X,\ov X}^S\shm$. As a consequence, still denoting $S':=\alpha^{-1}(S)$, we find an isomorphism up to $S$\nobreakdash-torsion
\[
\Conj^{S'}_{X,\ov X}(\DXSp\cdot\text{``}f^{\alpha(s')} u\text{''})\simeq\shd_{\XbS'/S'}\cdot\text{``}\ov f^{\alpha(s')} u\text{''}.
\]
For example, consider the case where $X\subset\CC$ with coordinate $z$, $f(z)=z$ and $u=\ov z^m$ for some $m\geq0$, so that $\ord(u)=0$. Then we set $S=\CC\moins\NN^*$, $S'=\alpha^{-1}(S)$ and~$v$ is the extension to $S'$ of $|z|^{2\alpha(s')}\ov z^m$ defined for $\reel\alpha(s')+m/2>-1$. Then $v$ is a global section of the sheaf $\Db_{\XS'/S'}$ satisfying $(z\partial z-\nobreak\alpha(s'))v=0$, so that $\DXSp/(z\partial z-\alpha(s'))\simeq\DXSp\cdot v$, and we have
\[
\Conj_{X,\ov X}^{S'}\bigl[\DXSp/(z\partial z-\alpha(s'))\bigr]\simeq\shd_{\XbS'/S'}/(\ov z\partial\ov z-\alpha(s')-m).
\]
\end{example}

\appendix
\section*{Appendix. The sheaf of relative distributions}
\refstepcounter{section}
\renewcommand{\thetheorem}{\Alph{section}.\arabic{theorem}}
\renewcommand{\theequation}{\Alph{section}.\arabic{equation}}
\setcounter{theorem}{0}
\setcounter{equation}{0}

In this appendix, we give details on the sheaf of partially holomorphic distributions $\Db_{\XS/S}$. This sheaf was already considered by Schapira and Schneiders in \cite{SS1} (\cf also \cite{Sabbah05,Mochizuki07}). The notion of relative current of maximal degree leads to the sheaf $\Cb_{M\times T/T}$ as mentioned in the introduction.

\subsection{Partially holomorphic distributions}
Let $M$ be a $C^\infty$ manifold of real dimension $m$ and let $T$ be a complex manifold (we also denote by~$T$ the underlying $C^\infty$ manifold). Let $p:M\times T\to T$ denote the projection. Let $W$ be an open set of $M\times T$. We say that a distribution $u\in\Db(W)$ is partially holomorphic with respect to~$T$, or $T$-holomorphic, if it satisfies the partial Cauchy-Riemann equation $\ov\partial_Tu=0$. The subsheaf of $\Db_{M\times T}$ consisting of $T$-holomorphic distributions is denoted by $\Db_{M\times T/T}$. For example, if $u\in\Db(M)$, then $u$ defines a section of $\Db_{M\times T/T}(M\times T)$. The following result is obtained by adapting the proof of \cite[Th.\,4.4.7]{Hormander03}.

\begin{proposition}\label{prop:horm}
Let $U$ be an open subset of $\RR^m$, $K$ a compact subset of~$T$ and let $u$ be a distribution defined on a neighbourhood of $U\times K$ in $U\times T$. If $\ov\partial_T u=0$, then there exists a neighbourhood $V$ of $K$ and finite family of functions $f_\alpha\in C^0(U\times V)$ which are $V$-holomorphic and a decomposition $u=\sum_\alpha\partial_x^\alpha f_\alpha$.
\end{proposition}

Let us mention some obvious consequences.
\begin{enumerate}
\item\label{enum:obviousa}
A partially holomorphic distribution $u\in\Db_{M\times T/T}(M\times T)$ extends as a continuous linear map from the space of $C^\infty$ functions on $M\times T$ which are $T$-holomorphic and have $p$-proper support to the space of holomorphic functions $\sho(T)$, endowed with the usual family of semi-norms (sup of absolute value on compact subsets of $T$).
\item\label{enum:obviousb}
As a consequence, a $T$-holomorphic distribution on $M\times T$ can be restricted to each $t_o\in T$, giving rise to a distribution $u_{|t_o}\in\Db(M)$. More generally, for any holomorphic function $g:T'\to T$, the pullback $g^*u$ is defined as a $T'$-holomorphic distribution on $M\times T'$.
\item\label{enum:obviousc}
If $T$ is connected and if two $T$-holomorphic distributions on $M\times T$ coincide on $M\times V$ for some nonempty open subset $V$ of $T$, they coincide: this follows from the first point above.
\item\label{enum:obviousd}
Let $f\in\sho(T)$ not vanishing on a dense open subset of $T$. For $u\in\Db_{M\times T/T}(M\times T)$, if $fu=0$, then $u=0$.
\end{enumerate}

\begin{corollary}\label{cor:expansion}
Assume that $T$ is an open polydisc centered at the origin with coordinates $t_1,\dots, t_n$. Let $u\!\in\!\Db_{M\times T/T}(M\times T)$. Then $u$ admits a termwise $T$-differentiable convergent expansion
\[
u=\sum_{\bmm\in\NN^n} u_{\bmm}t^{\bmm}
\]
where $u_{\bmm}$ in $\Db(M)$ (regarded in $\Db_{M\times T/T}(M\!\times\!T)$) are uniquely determined from $u$, and convergence meaning convergence in $\Db(M\times T)$.
\end{corollary}

\begin{proof}
We apply Proposition \ref{prop:horm} on each relatively compact open polydisc~$T'$ in $T$. Each $f_\alpha(x,t)$ is $T'$-holomorphic on $M\times T'$ and can be expanded as $\sum_{\bmm}f_{\alpha,\bmm}(x)t^{\bmm}$ with $f_{\alpha,\bmm}(x)$ continuous. The existence of the expansion follows, as well as its $T'$-differentiability. Uniqueness is then clear: for example, $u_\mathbf{0}=u_{|t=0}$. By uniqueness, applying the result to an increasing sequence of such subpolydiscs converging to $T$, we obtain the assertion.
\end{proof}

\begin{corollary}\label{memento}
For any $t_o\in T$, we have a natural identification
\[
Li^{*}_{t_o}\Db_{M\times T/T}\simeq \Db_M.
\]
\end{corollary}

\begin{proof}
Let us start by proving that the complex $Li^{*}_{t_o}\Db_{M\times T/T}$ is concentrated in degree zero. This a local statement and by induction on $\dim T$, it is enough to check that, for any local coordinate~$t$ on~$T$, $t: \Db_{M\times T/T}\to\Db_{M\times T/T}$ is injective. Let $u$ be a section of $\Db_{M\times T/T}$ in the neighbourhood of $(x_0,0)$ such that $tu=0$. According to \eqref{enum:obviousd} above we get $u=0$ hence the vanishing of $\shh^{-1} Li^{*}_{{t=0}}\Db_{M\times T/T}$.

That $i^*_{t_o}\Db_{M\times T/T}=\Db_M$ follows from \eqref{enum:obviousb} above.
\end{proof}

Let $\rho:T'\to T$ be a finite ramification around the coordinate axes, that we write $\rho(t'_1,\dots,t'_n)=(t^{\prime k_1}_1,\dots,t^{\prime k_n}_n)$. If $u$ is $T$-holomorphic, it follows from~\eqref{enum:obviousb} above that $\rho^*u$ is $T'$\nobreakdash-holo\-morphic. Moreover, the expansion of~$\rho^*u$ is obtained from that of $u$ as
\begin{equation}\label{eq:rhostaru}
\rho^*u=\sum_{\bmm} u_{\bmm}t^{\prime k_1a_1}\cdots t^{\prime k_na_n}.
\end{equation}

\begin{corollary}\label{cor:ramif}
Under these assumption, the natural morphism
\begin{align*}
\sho(T')\otimes_{\sho(T)}\Db_{M\times T/T}(M\times T)&\to\Db_{M\times T'/T'}(M\times T')\\
h\otimes u&\mto h\cdot\rho^*u
\end{align*}
is an isomorphism.
\end{corollary}

\begin{proof}
Surjectivity is obvious, by using the expansion of Corollary \ref{cor:expansion} on $M\times T'$ and \eqref{eq:rhostaru}. Since $\sho(T')$ is free over $\sho(T)$ with basis $t^{\prime\bmr}$, $0\leq r_i<k_i$ ($i=1,\dots,n$), injectivity follows from uniqueness of the decomposition in Corollary \ref{cor:expansion} on $M\times T'$.
\end{proof}

By sheafifying Corollary \ref{cor:ramif} we obtain:

\begin{corollary}\label{Cram}
Let $T,T'$ be polydiscs and let $\rho:T'\to T$ be a ramification along the coordinate axes. Then
\[
\rho^*\Db_{M\times T/T}:=\sho_{T'}\otimes_{\sho_T}\Db_{M\times T/T}\simeq\Db_{M\times T'/T'}.
\]
\end{corollary}

\subsection{Integration and pushforward of relative currents by a proper map}\label{subsec:currents}
The sheaf $\Cb_{M\times T/T}$ of $T$\nobreakdash-holo\-morphic currents of maximal degree is the subsheaf of the sheaf of degree $m$ currents which are killed by $\ov\partial_T$.

Let $f:M\to N$ be a $C^\infty$ map between $C^\infty$ manifolds. Let $W'\subset N\times B$ be an open set and let $W=f^{-1}(W')$. Let $u\in\Cb_{M\times T/T}(W)$ and assume that~$f$ is proper on $\Supp u$. Then the integral $\int_fu$, which is a current of degree $\dim N$ on $N\times T$, is also killed by $\ov \partial_T$, hence belongs to $\Cb_{N\times T/T}(W')$.

Let us now consider the sheaf-theoretic pushforward of $\Cb_{M\times B/B}$ by $f\times\id$.

\begin{proposition}\label{prop:Cacyclic}
The sheaf $\Cb_{M\times T/T}$ is $(f\times\id)_!$-acyclic.
\end{proposition}

\begin{proof}
For any $t_o\in T$, the sheaf-theoretic restriction $i_{t_o}^{-1}\Cb_{M\times T/T}$ is c-soft, being a $\shc^\infty_M$-module. It follows that for any $(y_o,t_o)\in N\times T$, denoting by $i:f^{-1}(y_o)\times\{t_o\}\hto M\times T$ the inclusion,
\[
i_{(y_o,t_o)}^{-1}\shh^j(f\times\id)_!\Cb_{M\times T/T}=H^j_\rc(f^{-1}(y_o)\times\{t_o\},i^{-1}\Cb_{M\times T/T})=0
\]
for any $j>0$.
\end{proof}

\backmatter
\bibliographystyle{amsplain}
\bibliography{relative-hermitian-duality}
\end{document}